\documentclass[abstracton]{scrartcl}
\usepackage[utf8]{inputenc}
\usepackage[english]{babel}
\usepackage{amssymb,amsbsy,amsmath,amsthm,mathtools}
\usepackage[round]{natbib}

\newtheorem{theorem}{Theorem}
\newtheorem{corollary}[theorem]{Corollary}
\newtheorem{lemma}{Lemma}
\theoremstyle{definition}\newtheorem{definition}{Definition}
\theoremstyle{remark}
\newcommand{\RR}{\mathbb R}
\newcommand{\NN}{\mathbb N}
\newcommand{\F}{\mathcal{F}}
\newcommand{\N}{\mathcal{N}}
\newcommand{\e}{\varepsilon}
\newcommand{\X}{\boldsymbol{X}}
\newcommand{\Z}{\boldsymbol{Z}}
\renewcommand{\l}{\ell}
\renewcommand{\L}{\boldsymbol{L}}
\newcommand\numberthis{\addtocounter{equation}{1}\tag{\theequation}}

\allowdisplaybreaks
\begin{document}

\title{The Function-Indexed Sequential Empirical Process under Long-Range Dependence}
\author{Jannis Buchsteiner\thanks{\texttt{jannis.buchsteiner@rub.de}
}\\ \normalsize{\textit{Fakultät für Mathematik, Ruhr-Universität Bochum, 
Germany.}}}
\date{}
\maketitle

\begin{abstract}
Let $(\X_j)_{j\geq1}$ be a multivariate long-range dependent Gaussian process. We study the asymptotic behavior of the corresponding sequential empirical process indexed by a class of functions. If some entropy condition is satisfied we have weak convergence to a linear combination of Hermite processes.  \\[1ex]
\end{abstract}

\section{Introduction}
For a real-valued stationary process $(Y_j)_{j\geq1}$ the sequential empirical process $(R_N(x,t))$ is defined by
\begin{equation}\label{equation classical sep}
 R_N(x,t):=\sum_{j=1}^{\lfloor Nt\rfloor}\left(1_{\{Y_j\leq x\}}-P(Y_j\leq x)\right),\hspace{5mm}x\in\RR,t\in[0,1].
\end{equation}
This process plays an important role in nonparametric statictics, for example in change-point analysis. If $(Y_j)_{j\geq1}$ is a subordinated Gaussian process which exhibits long-range dependence, weak convergence was established by \citet{Dehling}. The limiting process is given by the product of a deterministic function and an Hermite process $(Z_m(t))_{0\leq t\leq1}$. The latter can be represented as a stochastic integral in the spectral domain, more precisely
\begin{equation}\label{eq Hermite Prozess Darstellung}
Z_m(t)=K(m,D)\int_{\RR^m}^{''}\frac{e^{it(x_1+\ldots+x_m)}-1}{i(x_1+\ldots+x_m)}\prod_{j=1}^m|x_j|^{-(1-D)/2}\tilde{B}(dx_1)\cdots \tilde{B}(dx_m),
\end{equation}
where $\tilde{B}$ is a suitable random spectral measure and $K(m,D)$ is normalizing constant. The double prime indicates that integration excludes not only $\{x\in\RR^m:x^{(i)}=x^{(j)}, 1\leq i<j\leq m\}$ but also $\{x\in\RR^m:x^{(i)}=-x^{(j)}, 1\leq i<j\leq m\}$. For details and further representations see \citet{Taqqu1979} and \citet{Taqqu2010}. A first step in generalizing this result to multivariate observations was done by \citet{Marinucci}. He studied the asymptotics of the empirical process $(R_N(x,1))$ based on a two-dimensional long-range dependent process, where $\leq$ in \eqref{equation classical sep} is understood componentwise. However, this is nothing but checking if the observation lies inside a rectangle or not. Since there is no reason to restrict ourselves to such specific sets, it could be interessting to study indicators of balls or ellipsoids as well. This purpose yields the sequential empirical process indexed by a class of functions $\F$. 
\begin{definition}\label{definition function seq emp}
Let $(\X_j)_{j\geq 1}$ be stationary $\RR^p$-valued process and let $\F\subset L^2(P_{\X_1})$ be a class of square-integrable functions. Furthermore, let $\F$ be uniformly bounded, i.e. $\sup_{f\in\F}|f(x)|<\infty$ for all $x\in\RR^p$. We define the function-indexed sequential empirical process $(R_N(f,t))_{\F\times[0,1]}$ by
\begin{equation*}
R_N(f,t)=\sum_{j=1}^{\lfloor Nt\rfloor}f(\X_j)-E(f(\X_1)).
\end{equation*}
\end{definition}
Most asymptotic results one can find in the literature are formulated for the one-parameter process $(R_N(f,1))$. Under independence \citet{Dudley78} studied the empirical process indexed by a class of measurable sets, i.e. he considered $\F=\{1_{A}(\cdot):A\in\mathcal{A}\}$, where $\mathcal{A}$ is a suitable subset of the Borel $\sigma$-algebra. He stated different assumptions under which weak convergence to a Gaussian process holds, including a so-called metric entropy with inclusion. Generalizing this idea, \citet{Ossiander} introduced $L^2$-brackets to approximate the elements of $\F$. These brackets allow to study larger classes of functions as long as a metric entropy integrability condition is satisfied, see \citet[Theorem 3.1]{Ossiander}. A bracketing condition under strong mixing was stated by \citet{Andrews}. \citet{Doukhan} studied the function-indexed empirical process for $\beta$-mixing sequences. The case of Gaussian long-range dependent random vectors was already handled by \citet[Theorem 9]{Arcones}. The assumption on the bracketing number therein is very restrictive and will be considerably improved in this paper. Note, the only known result for the two-parameter process $(R_N(f,t))$ was given by \citet{Tusche2} for multiple mixing data.

The simplest way to define multivariate long-range dependence for an $\RR^p$-valued process $(\X_j)_{j\geq1}$, $\X_j=(X_j^{(1)},\ldots,X_j^{(p)})$, is the following one. Assume that the component-processes $(X_j^{(1)}),\ldots,(X_j^{(p)})$ are independent and that each of them individually satisfies the well-known time-domain long-range dependence condition that is
\begin{equation*}
r^{(i)}(k)=\operatorname{Cov}(X_1^{(i)},X_{1+k}^{(i)})=k^{-D_i}L_i(k), \hspace{5mm}1\leq i\leq p.
\end{equation*}
For $p=2$ this concept was used by \citet{Marinucci} and, if $L_1=L_2$ and $D_1=D_2$, by \citet{Taufer}. However, using this approach we can not model dependency between different components. Therefore, we have to replace the assumption of independent components by a suitable cross-covariance structure. For $1\leq i,j\leq p$ we call the function 
\begin{equation*}
r^{(i,j)}(k)=\operatorname{Cov}(X_1^{(i)},X_{1+k}^{(j)})
\end{equation*}
a cross-covariance function. \citet{Ho} considered a stationary bivariate Gaussian process such that for $1\leq i,j\leq2$, $r^{(i,j)}(k)$ is asymptotic equal to $k^{-\beta_{i,j}}$, $\beta_{i,j}>0$. The definition of multivariate long-range dependence we use here can also be labeled as a \textit{single-parameter long-range dependence}. It was originally stated by \citet[p. 2259]{Arcones}, see also \citet[p. 300]{Beran}, and is included in the very general concept of multivariate long-range dependence given by \citet[Definition 2.1]{Pipiras}.
\begin{definition}\label{Definition multivariate LRD}
A stationary $\RR^p$-valued process $(\X_j)_{j\geq1}$ with finite second moments is called long-range dependent if for each $1\leq i,j\leq p$
\begin{equation}\label{eq cross}
r^{(i,j)}(k)=c_{ij}k^{-D}L(k)\hspace{10mm}\textnormal{for }k\geq1,
\end{equation}
where $L$ is slowly varying at infinity, $0<D<1$, and $c_{ij}\in\RR$ are not all equal to zero.
\end{definition}
In the definition given by \citet{Pipiras} the long-range dependence parameter $D$ also depends on $i$ and $j$ so that the cross-covariance is given by
\begin{equation}\label{eq cross alternativ}
r^{(i,j)}(k)=c_{ij}k^{-(D_i+D_j)/2}L(k)\hspace{10mm}0<D_i,D_j<1.
\end{equation} 
Although this assumption is more realistic, multiple different parameters yield a degenerate limiting behavior, see \citet[p. 308]{Leonenko}. We briefly discuss this phenomenon subsequent to Theorem \ref{Theorem Reduction Functions}.

\section{Assumptions and Techniques}
From the probabilistic point of view we will interpret $(R_N(f,t))$ as a random element in $\l^{\infty}(\F\times[0,1])$, that is the space of real-valued bounded functions on $\F\times[0,1]$. The uniformly boundedness of $\F$ stated in Definition \ref{definition function seq emp} ensures well-definedness. We equip $\l^{\infty}(\F\times[0,1])$ with the supremum norm and the corresponding Borel $\sigma$-algebra. In general, empirical processes are not measurable with respect to the Borel $\sigma$-algebra generated by the uniform metric, see \citet[p. 157]{Billingsley}. \citet{Dehling} handled this problem by considering the open ball $\sigma$-algebra instead of the Borel one. We will use the idea of outer expectation introduced by Hoffmann-Jørgensen. For the sake of completeness we recall the definition. A detailed introduction to this concept can be found in \citet[chapter 1]{vanderVaart}. 
\begin{definition}\label{def outer}
Let $(\Omega, \mathcal{A},P)$ be an arbitrary probability space and let $\hat{\RR}$ be the extended real line equipped with the Borel $\sigma$-algebra.
\begin{itemize}
\item[\textit{i)}] For any $B\subset\Omega$ we call
\begin{equation*}
P^*(B):=\inf\left\{P(A):\hspace{2mm}A\supset B, A\in\mathcal{A}\right\}
\end{equation*}
the outer probability.
\item[\textit{ii)}] For any map $X:\Omega\rightarrow\hat{\RR}$ we call
\begin{equation*}
E^*X:=\inf\left\{EY:\hspace{2mm}Y\geq X, Y:\Omega\rightarrow\hat{\RR} \textnormal{ measurable and $EY$ exists}\right\}
\end{equation*}
the outer integral of $X$.
\item[\textit{iii)}] A sequence of maps $(X_n)_{n\geq1}$, taking values in a metric space $S$, converges weakly to a Borel-measurable random variable $X:\Omega\rightarrow S$ if
\begin{equation*}
E^*f(X_n)\longrightarrow Ef(X)
\end{equation*}
for all continuous, bounded functions $f:S\rightarrow\RR$.
\end{itemize}
\end{definition}
If measurability holds, part \textit{iii)} of the above definition conforms with classical weak convergence of random variables. Moreover, one can show that the Portmanteau theorem,  continuous mapping theorem, Prohorov's theorem, and other well-known results are still applicable in the context of outer weak convergence, see \citet[section 1.3]{vanderVaart}. Convergence in outer probability and outer almost sure convergence can be defined analogous.

 From now on assume that the process $(\X_j)_{j\geq1}$ is standard normal and long-range dependent in the sense of Definition \ref{Definition multivariate LRD}, i.e.
\begin{align*}
EX^{(i)}_1=&0\hspace{5mm}1\leq i \leq p\\
EX^{(i)}_1X^{(j)}_1=&\delta_{ij}\numberthis\label{equation components uncorrelated}\\
r^{(i,j)}(k)=&c_{ij}L(k)k^{-D}\hspace{10mm}\textnormal{for }k\geq1.
\end{align*}
The existence of such a process is ensured because there exist a linear representation for multivariate long-range dependent processes, see \citet[Corollary 4.1]{Pipiras}. Condition \eqref{equation components uncorrelated} is not restrictive, since for an arbitrary covariance matrix $\Sigma$ we have $G\X_1\sim\N(0,\Sigma)$, where $G$ is the Cholesky decomposition of $\Sigma$. In this case one can study $(R_N(f\circ G,t))$ instead of $(R_N(f,t))$.

By $\N_p$ we denote a $p$-dimensional standard normal distribution and by $H_k$ the $k$-th Hermite polynomial given by
\begin{equation*}
 H_k(x):=(-1)^ke^{x^2/2}\frac{d^k}{dx^k}e^{-x^2/2}.
\end{equation*}
Using these univariate polynomials, we define multivariate Hermite polynomials by
\begin{equation*}
H_{l_1,\ldots,l_p}(x):=H_{l_1}(x^{(1)})H_{l_2}(x^{(2)})\cdots H_{l_p}(x^{(p)}),\hspace{10mm}x\in\RR^p.
\end{equation*}
The collection of all multivariate Hermite polynomials of order $p$ forms an orthogonal basis of $L^2(\N_p)$, see \citet[p. 122]{Beran}. Thus, for any $f\in \F\subset L^2(\N_p)$ we have a $L^2$-series expansion of $f(\X_j)-E(f(\X_1))$ namely
\begin{equation*}
f(\X_j)-E(f(\X_1))=\sum_{k=m(f)}^{\infty}\sum_{l_1+\ldots+l_p=k}\frac{J_{l_1,\ldots,l_p}(f)}{l_1!\cdots l_p!}H_{l_1,\ldots,l_p}(\X_j).
\end{equation*}
The index $m(f)\geq1$ denotes the order of the first non-zero Hermite coefficient. We call the minimum $m:=\min\{m(f):f\in\F\}$ the Hermite rank of $\F$. Note, the Hermite rank $m$ is only determined by $f$ and not by a specific choice of $c_{ij}$.

Next we must account for the fact that the elements of $\F$ can be well approximated based on a relatively small number of functions. The definition of $L^2$-brackets we give next based on that of \citet{Ossiander}.
\begin{definition}
\begin{itemize}
\item[\textit{i)}] For $\e>0$ and $l,u\in\F$ with $l\leq u$ we call $\{f\in\F:l\leq f\leq u\}$ a $\e$-bracket if $\|u-l\|_2\leq\e$, where $\|\cdot\|_2$ denotes the $L^2$-norm with respect to $\N_p$.
\item[\textit{ii)}]The smallest number $N(\e)$ of $\e$-brackets needed to cover $\F$ is called bracketing number.
\end{itemize}
\end{definition}
Obviously, the bracketing number increases if $\e$ tends to zero. Therefore, we assume that the following entropy condition holds
\begin{equation}\label{B}
\int\limits_0^1\e^{r-1}N(\e)^2d\e<\infty~~~~\textnormal{ for some integer } r\geq 1.
\end{equation}
Roughly speaking, condition \eqref{B} is satisfied as long as $N(\e)$ grows polynomial in $\e^{-1}$. Although under independence one can handle exponential growth rates, see \citet[Theorem 3.1]{Ossiander}, the entropy condition we require here is not untypical for the dependent case. Similar conditions were stated by \citet[Theorem 2.2]{Andrews} and \citet[Theorem 2.5]{Tusche2}. Besides, \eqref{B} is a much weaker assumptions than the one stated by \citet[(5.1)]{Arcones} namely
\begin{equation*}
\int\limits_0^1N(\e)d\e<\infty.
\end{equation*}

Additionally, we need uniformly bounded $2q$-th moments for all $f\in\F$, where $q$ depends on the Hermite rank $m$, the long-range dependence parameter $D$, and $r$ from \eqref{B}. More precisely,
\begin{equation}\label{C}
\sup_{f\in\F}E(f(\X_1))^{2q}<\infty,\hspace{10mm}\textnormal{ where }2q>\frac{mDr}{1-mD}\vee mr.
\end{equation}
This choice of range for $2q$ is determined by technical constraints arising in the proof of Lemma \ref{lemma 2q-te Momente}. The lower bound is $mr$, if and only if $D<1/(m+1)$.

\section{Results}
\begin{theorem}\label{Theorem Limit Functions}
  Let $(\X_j)_{j\geq1}$ be a $p$-dimensional standard Gaussian long-range dependent process in the sense of Definition \ref{Definition multivariate LRD} and let $\F$ be a uniformly bounded class of functions satisfying \eqref{B} and \eqref{C}. Moreover, let $0<D<1/m$, where $m$ is the Hermite rank of $\F$, and set 
 \begin{equation*}
d_N^2=\operatorname{Var}\left(\sum_{j=1}^N H_m(X_j^{(1)})\right).
\end{equation*}
 Then 
 \begin{equation*}
 \left\{d_N^{-1}R_N(f,t): (f,t)\in\F\times [0,1]\right\}\phantom{\Biggl(}
\end{equation*}
converges weakly in $\l^{\infty}(\F\times [0,1])$ in the sense of Definition \ref{def outer}, to
 \begin{multline}\label{equation limiting process}
 \Biggl\{\sum_{j_1,\ldots,j_m=1}^p \tilde{J}_{j_1,\ldots,j_m}(f)\tilde{K}_{j_1,\ldots,j_m}(m,D)
 \int_{\RR^m}^{''}\frac{e^{it(x_1+\ldots+x_m)}-1}{i(x_1+\ldots+x_m)}\prod_{j=1}^m|x_j|^{-(1-D)/2}\\\tilde{B}^{(j_1)}(dx_1)\cdots \tilde{B}^{(j_m)}(dx_m): 
(f,t)\in\F\times [0,1]\Biggr\} .
\end{multline}
The measures $\tilde{B}^{(1)},\ldots,\tilde{B}^{(p)}$ are suitable Hermitian Gaussian random measures, \linebreak $\tilde{K}_{j_1,\ldots,j_m}(m,D)$ is a normalizing constant, $1\leq j_1,\ldots,j_m \leq p$, and
\begin{equation*}
\tilde{J}_{j_1,\ldots,j_m}(f)=(m!)^{-1}E\left(f(\X_1)\prod_{i=1}^p H_{i(j_1,\ldots,j_m)}(X_1^{(i)})\right),
\end{equation*}
where $i(j_1,\ldots,j_m)$ is the number of indices $j_1,\ldots,j_m$ that are equal to $i$.
\end{theorem}
The limiting process \eqref{equation limiting process} can be interpreted as a generalization of the univariate limit studied by \citet[Theorem 1.1]{Dehling}, especially the integrand is the same. But due to the fact that each component of $(\X_j)_{j\geq 1}$ contributes to the limit, we are faced with $p$ different integrators. More precisely, the above sum includes all possibilities to generate a $m$-fold stochastic integral using $p$ Gaussian random measures. Such integrals were initially studied by \citet[section 2]{Fox} in the case of real-valued Gaussian random measures. The process \eqref{equation limiting process} appeared first time in a paper by \citet[Theorem 6]{Arcones}. He studied Theorem \ref{Theorem Limit Functions} in the non-uniform case, particularly, Arcones proved  a functional non-central limit theorem for $(f(\X_j))_{j\geq1}$, $f\in L^2(\N_p)$. We have corrected the domain of integration, i.e. we replaced $[-\pi,\pi]^m$ by $\RR^m$, see also \citet[Theorem 4.22]{Beran}. 

\citet[Section 3]{Tusche} calculated the bracketing number for a few different classes of functions including indicators of hyperrectangles. In all cases the bracketing number grows polynomially so that condition \eqref{B} is fulfilled. Thus, Theorem \ref{Theorem Limit Functions} implies a multivariate version of Theorem 1.1 by \citet{Dehling}.
\begin{corollary}\label{Corollary Limit Functions}
  Let $(\X_j)_{j\geq1}$ be a $p$-dimensional standard Gaussian long-range dependent process in the sense of Definition \ref{Definition multivariate LRD} and let $G:\RR^p\rightarrow\RR^k$ be a measurable function. Moreover, let $0<D<1/m$, where $m$ is the Hermite rank of $\F=\{1_{\{G(\cdot)\leq x\}}:x\in\RR^k\}$. Then 
 \begin{equation*}
 \left(d_N^{-1}\sum_{j=1}^{\lfloor Nt\rfloor}\left(1_{\{G(\X_j)\leq x\}}-P(G(\X_1)\leq x)\right)\right)\phantom{\Biggl(}
\end{equation*}
converges weakly in $\l^{\infty}(\F\times [0,1])$ in the sense of Definition \ref{def outer}, to \eqref{equation limiting process}.
\end{corollary}
We prove Theorem \ref{Theorem Limit Functions} by establishing a weak uniform reduction principle. The formulation is similiar to Theorem 3.1 by \citet{Dehling}. The notation $P^*$ denotes outer probability.
\begin{theorem}[Reduction principle]\label{Theorem Reduction Functions}
Under the assumptions of Theorem \ref{Theorem Limit Functions} there exist constants $C, \kappa>0$ such that for any
$0<\e\leq1$
\begin{align}
 P^*\Biggl(&\max_{n\leq N}\sup_{f\in\F}d_N^{-1}\left|\sum_{j=1}^n\left(f(\X_j)-E(f(\X_1))-\hspace{-3pt}
 \sum_{l_1+\ldots+l_p=m}\frac{J_{l_1,\ldots,l_p}(f)}{l_1!\cdots l_p!}H_{l_1,\ldots,l_p}(\X_j)\right)\right|
         >\varepsilon\Biggr)\notag\\
\leq &CN^{-\kappa}(1+\varepsilon^{-(2q+r+1)}).\notag
\end{align}
\end{theorem}
It is well known that $d_N$ behaves asymptotically like $N^{1-mD/2}L^{m/2}$, see \citet[Theorem 3.1]{Taqqu75}. If the cross-covariance is  given by \eqref{eq cross alternativ} instead of \eqref{eq cross} one has to normalize with $\operatorname{Var}\left(\sum_{j=1}^N H_m(X_j^{(k)})\right)$, where $D_k=\min\{D_i:1\leq i\leq p\}$. Consequently , only $\sum_{j=1}^{\lfloor Nt\rfloor} H_m(X_j^{(k)})$ contributes to the limit because all other terms tend to zero. This behavior was already studied by \citet[Theorem 1b]{Leonenko} and reasoned why we focus on a single long-range dependence parameter only. Note, all lemmas in the following section remain valid if we would study \eqref{eq cross alternativ} instead of \eqref{eq cross}, only the proof of Lemma \ref{lemma 2q-te Momente} has to be slightly modified.

\section{Proofs}
The proof of Theorem \ref{Theorem Reduction Functions} is organized as follows. First we prove a moment inequality for partial sums of subordinated Gaussian random vectors (Lemma \ref{lemma Taqqu1977}). We can use this result to control the increments of the reduced function-indexed process (Lemma \ref{lemma 2q-te Momente}). This is essential for proving Lemma \ref{lemma chaining function}. 

The following lemma by \citet{Bardet} is the multivariate version of Lemma 4.5 by \citet{Taqqu77}. The term $\e$-\textit{standard} used therein denotes a collection of standardized Gaussian random vectors,  whose cross-covariance function is bounded by $\e$. Furthermore, $\sum\hspace*{-0.5mm}{'}$ is the sum over all different indices $1\leq t_i\leq n$, $1\leq i\leq p'$, such that $t_i\neq t_j$ for $i\neq j$.
\begin{lemma}[{\citet[Lemma 1]{Bardet}}]\label{lemma bardet}
Let $(\X_1,\ldots,\X_n)$ be a $\e$-standard Gaussian vector, $\X_t=(X_t^{(1)},\ldots,X_t^{(p)})\in\RR^{p}$, $p\geq 1$, and consider a set of functions $f_{j,t,n}\in L^2(\N_p)$, $1\leq j\leq p'$, $p'\geq 2$, $1\leq t\leq n$. For given integers $m\geq1$, $0\leq \alpha\leq p'$, $n\geq1$, define
\begin{equation*}
Q_n:=\max_{1\leq t\leq n}\sum_{1\leq s\leq n,s\neq t}\max_{1\leq u,v\leq p}\left|EX_t^{(u)}X_s^{(v)}\right|^m.
\end{equation*}
Assume that the functions $f_{1,t,n},\ldots,f_{\alpha,t,n}$ have a Hermite rank at least equal to $m$ for any $n\geq 1$, $1\leq t\leq n$, and that
\begin{equation*}
\e<\frac{1}{pp'-1}.
\end{equation*}
Then
\begin{equation*}
\sum_{t_1,\ldots,t_{p'}=1}^n\hspace*{-5mm}{'}~~\left|E\left(f_{1,t_1,n}(\X_{t_1})\cdots f_{p',t_{p'},n}(\X_{t_{p'}})\right)\right|\leq C(\e,p',m,\alpha,p) Kn^{p'-\alpha/2}Q_n^{\alpha/2},
\end{equation*}
where the constant $C(\e,p',m,\alpha,p)$ depends on $\e,p',m,\alpha,p$ only, and
\begin{equation*}
K=\prod_{j=1}^{p'}\max_{1\leq t\leq n}\|f_{j,t,n}\|_2.
\end{equation*}
\end{lemma}
For our purposes it is enough to consider functions $f_{j,t,n}$ which do not depend on $t$ and $n$, i.e. $f_{j,t,n}=g_j$. In this case the connection to the result of \citet[Lemma 4.5]{Taqqu77} becomes clearer. An advantage of the above bound is that the constant on the right-hand side is separated into $C(\e,p',m,\alpha,p)$ and $K$. This detail will help us to prove the following lemma.
\begin{lemma}\label{lemma Taqqu1977}
Let $(\X_j)_{j\geq1}$ be a $p$-dimensional standard Gaussian process which exhibits long-range dependence in the sense of Definition \ref{Definition multivariate LRD} and let $r^{(i,j)}(k)\leq 1/(4pq-1)$ for all $k\in\NN$, $1\leq i,j\leq p$ and some $q\in\NN$. Furthermore, let $g:\RR^p\rightarrow\RR$ be a measurable function such that the Hermite rank of $g$ is at least $m$ and $E(g(\X_1))^{2q}<\infty$. If $0<D<1/m$ then for all $n\in\NN$ we have
\begin{equation*}
E\left(\sum_{j=1}^ng(\X_j)\right)^{2q}\leq C\left(1\vee \left(Eg(\X_1)^{2q}\right)^q\right)\left(n\sum_{k=0}^n \max_{1\leq i,j\leq p}(r^{(i,j)}(k))^m\right)^{q}.
\end{equation*}
\end{lemma}
\begin{proof}
We follow the lines of the proof given by \citet[Proposition 4.2.]{Taqqu77}. Using the multinomial theorem we get
\begin{align*}
&E\left(\sum_{j=1}^ng(\X_j)\right)^{2q}\\
=&\sum_{\substack{k_1+\ldots+k_n=2q\\k_1,\ldots,k_n\geq1}}\frac{(2q)!}{k_1!\cdots k_n!}Eg^{k_1}(\X_1)\cdots g^{k_n}(\X_n)\\
=&\sum_{p'=1}^{2q}\sum_{1\leq u_1<\ldots<u_{p'}\leq n}\sum_{\substack{k_{u_1}+\ldots+k_{u_{p'}}=2q\\k_{u_1},\ldots,k_{u_{p'}}\geq1}}
\frac{(2q)!}{k_{u_1}!\cdots k_{u_{p'}}!}Eg^{k_{u_1}}(\X_{u_1})\cdots g^{k_{u_{p'}}}(\X_{u_{p'}})\\
\leq &C\sum_{p'=1}^{2q}\max_{\substack{v_1+\ldots+v_{p'}=2q\\v_1,\ldots,v_{p'}\geq1}}\sum_{u_1,\ldots,u_{p'}=1}^n\hspace*{-5mm}{'}~~
\left|Eg^{v_1}(\X_{u_1})\cdots g^{v_{p'}}(\X_{u_{p'}})\right|.\numberthis\label{eq Taqqu77 two cases}
\end{align*}
Let us distinguish two cases,  $p'\leq q$ and $p'>q$, starting with the first one. For $v_1,\ldots,v_{p'}\geq1$ satisfying $v_1+\ldots+v_{p'}=2q$ we have by Hölder's inequality
\begin{align*}
\left|Eg^{v_1}(\X_{u_1})\cdots g^{v_{p'}}(\X_{u_{p'}})\right|
\leq&\prod_{i=1}^{p'}\prod_{j=1}^{v_i}\left(Eg(\X_{u_i})^{2q}\right)^{1/2q}\\
=&\prod_{i=1}^{p'}\left(Eg(\X_{1})^{2q}\right)^{v_i/2q}\\
=&Eg(\X_{1})^{2q}.
\end{align*}
Thus,
\begin{align*}
&\max_{\substack{v_1+\ldots+v_{p'}=2q\\v_1,\ldots,v_{p'}\geq1}}\sum_{u_1,\ldots,u_{p'}=1}^n\hspace*{-5mm}{'}~~
\left|Eg^{v_1}(\X_{u_1})\cdots g^{v_{p'}}(\X_{u_{p'}})\right|\\
\leq& n^{p'}Eg(\X_{1})^{2q}\\
\leq& Eg(\X_{1})^{2q}\left(n\sum_{k=0}^n \max_{1\leq i,j\leq p}(r^{(i,j)}(k))^m\right)^q,
\end{align*}
since $p'<q$. To handle the second case we have to ensure that Lemma \ref{lemma bardet} can be applied. If  $p'>q$ and $v_1+\ldots+v_{p'}=2q$ then it is impossible that all indices $v_1,\ldots,v_{p'}$ are greater than one. We separate those which are equal to one as follows
\begin{align*}
&\max_{\substack{v_1+\ldots+v_{p'}=2q\\v_1,\ldots,v_{p'}\geq1}}\sum_{u_1,\ldots,u_{p'}=1}^n\hspace*{-5mm}{'}~~
\left|Eg^{v_1}(\X_{u_1})\cdots g^{v_{p'}}(\X_{u_{p'}})\right|\\
=&\max_{\alpha_{\min}\leq\alpha\leq\alpha_{\max}}\max_{\substack{v_{\alpha+1}+\ldots+v_{p'}=2q\\v_{\alpha+1},\ldots,v_{p'}\geq2}}
\sum_{u_1,\ldots,u_{p'}=1}^n\hspace*{-5mm}{'}~~
\left| Eg(\X_{u_1})\cdots g(\X_{u_{\alpha}}) g^{v_{\alpha+1}}(\X_{u_{\alpha+1}})\cdots g^{v_{p'}}(\X_{u_{p'}})\right|.
\end{align*}
The numbers $\alpha_{\min}$ and $\alpha_{\max}$ describe the minimum and maximum number of indices that could be equal to one. More precisely,  $\alpha_{\min}=2p'-2q$ and
\begin{equation*}
\alpha_{\max}=\begin{cases}2q\hspace{10mm}&\textnormal{if }p'=2q\\p'-1\hspace{10mm}&\textnormal{otherwise.}\end{cases}
\end{equation*} 
For fixed $\alpha$ and $v_{\alpha+1},\ldots,v_{p'}$ we set $g_i(\cdot)=g(\cdot)$, if $1\leq i\leq\alpha$, and $g_i(\cdot)=g^{v_i}(\cdot)$, if $\alpha+1\leq i\leq p'$. Since $\max(v_{\alpha+1},\ldots,v_{p'})\leq q$, we have $g_i\in L^2(\N_p)$ for all $1\leq i\leq p'$. The Hermite rank of $g_1,\ldots,g_{\alpha}$ is at least $m$ and $r^{(i,j)}(k)\leq1/(4pq-1)<1/(pp'-1)$ for all $k\in\NN$. Therefore, the assumptions of Lemma \ref{lemma bardet} are satisfied with $\e=1/(4pq-1)$ and so
\begin{align*}
&\sum_{u_1,\ldots,u_{p'}=1}^n\hspace*{-5mm}{'}~~\left| Eg_1(\X_{u_1})\cdots  g_{p'}(\X_{u_{p'}})\right|\\
\leq&C(q,p',m,\alpha,p)\left(\prod_{i=1}^{p'}\|g_i\|_2\right)n^{p'-\alpha/2}\left(\sum_{k=0}^n \max_{1\leq i,j\leq p}(r^{(i,j)}(k))^m\right)^{\alpha/2}\\
\leq&C(q,p',m,\alpha,p)\left(1\vee \left(Eg(\X_1)^{2q}\right)^q\right)\left(n\sum_{k=0}^n \max_{1\leq i,j\leq p}(r^{(i,j)}(k))^m\right)^{q},
\end{align*}
since $p'-\alpha/2\leq p'-\alpha_{\min}/2=q$ and $\alpha/2\leq\alpha_{\max}/2\leq q$. With respect to \eqref{eq Taqqu77 two cases} we get
\begin{align*}
&E\left(\sum_{j=1}^ng(\X_j)\right)^{2q}\\
\leq &C\sum_{p'=1}^{2q}\max_{\substack{v_1+\ldots+v_{p'}=2q\\v_1,\ldots,v_{p'}\geq1}}\sum_{u_1,\ldots,u_{p'}=1}^n\hspace*{-5mm}{'}~~
\left|Eg^{v_1}(\X_{u_1})\cdots g^{v_{p'}}(\X_{u_{p'}})\right|\\
\leq&\begin{multlined}[t][\textwidth-4mm]C\left(n\sum_{k=0}^n \max_{1\leq i,j\leq p}(r^{(i,j)}(k))^m\right)^{q}
\Biggl(qEg(\X_1)^{2q}+\\
\sum_{p'=q+1}^{2q}\max_{\alpha_{\min}\leq\alpha\leq\alpha_{\max}}\max_{\substack{v_{\alpha+1}+\ldots+v_{p'}=2q\\v_{\alpha+1},\ldots,v_{p'}\geq2\\v_1=\cdots=v_{\alpha}=1}}C(q,p',m,\alpha,p)\left(1\vee \left(Eg(\X_1)^{2q}\right)^q\right)\Biggr)
\end{multlined}\\
\leq&C\left(1\vee \left(Eg(\X_1)^{2q}\right)^q\right)\left(n\sum_{k=0}^n \max_{1\leq i,j\leq p}(r^{(i,j)}(k))^m\right)^{q},
\end{align*}
where $C$ only depends on $p$, $q$ and $m$.
\end{proof}
To simplify the notation we set
\begin{equation*}
S_N(n,f)=d_N^{-1}\sum_{j=1}^n\left(f(\X_j)-Ef(\X_1)-\sum_{l_1+\ldots+l_p=m}\frac{J_{l_1,\ldots,l_p}(f)}{l_1!\cdots l_p!}H_{l_1,\ldots,l_p}(\X_j)\right)
\end{equation*}
and $S_N(n,g,h)=S_N(n,h)-S_N(n,g)$.
\begin{lemma}\label{lemma 2q-te Momente}
Let the assumptions of Theorem \ref{Theorem Limit Functions} be satisfied and assume further that $r^{(i,j)}(k)\leq 1/(4pq-1)$ for all $k\in\NN$, $1\leq i,j\leq p$. Then there exist positive constants $C$ and $\gamma$ such that for all $g,h\in\F$
\begin{equation*}
\left(Nd_N^{-1}\right)^rE\left(S_N(n,g,h)\right)^{2q}\leq C\left(\frac{n}{N}\right)N^{-\gamma}.
\end{equation*}
\end{lemma}
\begin{proof}
For all $g,h\in\F$ the Hermite rank of 
\begin{equation*}
h(X_j) -g(X_j)-E(h(X_1) -g(X_1))-\sum_{l_1+\ldots+l_p=m}\frac{J_{l_1,\ldots,l_p}(h-g)}{l_1!\cdots l_p!}H_{l_1,\ldots,l_p}(X_j)
\end{equation*}
is at least $m+1$. By assumption \eqref{C} the $2q-th$ moment of these functions are uniformly bounded. Therefore, we can apply Lemma \ref{lemma Taqqu1977}. Remember that  $d_N^2\in\mathcal{O}(N^{2-mD}L^m(N))$, see \citet[Theorem 3.1]{Taqqu75}, 
 and that
\begin{equation*}
n\sum_{k=0}^{n}(k^{-D}L(k))^{m+1}\leq Cn^{1\vee (2-(m+1)D)}L'(n),
\end{equation*}
where $L'$ is slowly varying at infinity, see \citet[p. 1777]{Dehling}. For simplicity products of slowly varying functions in $n$ and $N$ will combined into $\L(n,N)$. Lemma \ref{lemma Taqqu1977} yields
\begin{align*}
E\left(S_N(n,g,h)\right)^{2q}&\leq Cd_N^{-2q}\left(n\sum_{k=0}^n \max_{1\leq i,j\leq p}(r^{(i,j)}(k))^{m+1}\right)^{q}\\
&\leq Cd_N^{-2q}\left(n\sum_{k=0}^n (k^{-D}L(k))^{m+1}\right)^{q}\\
&\leq Cd_N^{-2q}n^{q\vee q(2-(m+1)D)}\L(n,N)\\
&\leq CN^{q(mD-2)}n^{q\vee q(2-(m+1)D)}\L(n,N)\\
&\leq C\left(\frac{n}{N}\right)^{q\vee q(2-(m+1)D)} N^{q(mD-2)}N^{q\vee q(2-(m+1)D)}\L(n,N)\\
&\leq C\left(\frac{n}{N}\right)N^{q(mD-1)\vee -qD}\L(n,N).
\end{align*}
Using the lower bound for $2q$ from \eqref{C} we can conclude the proof, since
\begin{align*}
&(Nd_N^{-1})^rC\left(\frac{n}{N}\right)N^{q(mD-1)\vee -qD}\L(n,N)\\
\leq& CN^{mDr/2}\left(\frac{n}{N}\right)N^{q(mD-1)\vee -qD}\L(n,N)\\
\leq& C\left(\frac{n}{N}\right)N^{-\gamma}
\end{align*}
for an appropriate small $\gamma>0$.
\end{proof}
\begin{lemma}\label{lemma chaining function}
Let the assumptions of Theorem \ref{Theorem Limit Functions} be satisfied and assume further that $r^{(i,j)}(k)\leq 1/(4pq-1)$ for all $k\in\NN$, $1\leq i,j\leq p$. Then there exist positive constants $\rho$ and $C$ such that for all $n\leq N$ and $0<\e\leq1$,
\begin{equation*}
P^*\left(\sup_{f\in\F}\left|S_N(n,f)\right|>\e\right)\leq CN^{-\rho}\left(\left(\frac{n}{N}\right)\e^{-(2q+r+1)}+\left(\frac{n}{N}\right)^{2-mD}\right).
\end{equation*}
\end{lemma}
\begin{proof}
For each $f\in\F$ and each $k\leq K$, where $K$ will be specified in \eqref{K}, we can find functions $l_{i_k(f)}^{(k)}$ and $u_{i_k(f)}^{(k)}$ such that 
$l_{i_k(f)}^{(k)}\leq f\leq u_{i_k(f)}^{(k)}$ and $\|u_{i_k(f)}^{(k)}-l_{i_k(f)}^{(k)}\|_2\leq 2^{-k}$. Using the lower functions, we get the telescoping sum
\begin{align*}
S_N(n,f)=&S_N(n,l_{i_0(f)}^{(0)})+S_N(n,l_{i_0(f)}^{(0)},l_{i_1(f)}^{(1)})+S_N(n,l_{i_1(f)}^{(1)},l_{i_2(f)}^{(2)})+\\
&\ldots+S_N(n,l_{i_{K-1}(f)}^{(K-1)},l_{i_K(f)}^{(K)})+S_N(n,l_{i_K(f)}^{(K)},f).
\end{align*}
The last term can be handled as follows
\begin{align*}
&{}\left|S_N(n,l_{i_K(f)}^{(K)},f)\right|\\
=&\begin{multlined}[t][\textwidth-4mm]\Biggl|d_N^{-1}\sum_{j=1}^n\Biggl(f(\X_j)-l_{i_K(f)}^{(K)}(\X_j)-E\left(f(\X_1)-l_{i_K(f)}^{(K)}(\X_1)\right)\\
-\sum_{l_1+\ldots+l_p=m}\frac{J_{l_1,\ldots,l_p}(f)-J_{l_1,\ldots,l_p}(l_{i_K(f)}^{(K)})}{l_1!\cdots l_p!}H_{l_1,\ldots,l_p}(\X_j)\Biggr)\Biggr|\end{multlined}\\
\leq&\begin{multlined}[t][\textwidth-4mm]d_N^{-1}\sum_{j=1}^n\left(f(\X_j)-l_{i_K(f)}^{(K)}(\X_j)+E\left(f(\X_1)-l_{i_K(f)}^{(K)}(\X_1)\right)\right)\\
+\sum_{l_1+\ldots+l_p=m}\left|\frac{E\left(\left(f(\X_1)-l_{i_K(f)}^{(K)}(\X_1)\right)H_{l_1,\ldots,l_p}(\X_1)\right)}{l_1!\cdots l_p!}\right|
d_N^{-1}\left|\sum_{j=1}^nH_{l_1,\ldots,l_p}(\X_j)\right|\end{multlined}\\
\leq&\begin{multlined}[t][\textwidth-4mm]d_N^{-1}\sum_{j=1}^n\left(u_{i_K(f)}^{(K)}(\X_j)-l_{i_K(f)}^{(K)}(\X_j)+
E\left(u_{i_K(f)}^{(K)}(\X_1)-l_{i_K(f)}^{(K)}(\X_1)\right)\right)\\
+\hspace{-3pt}\sum_{l_1+\ldots+l_p=m}\left|\frac{E\left(\left|u_{i_K(f)}^{(K)}(\X_1)-l_{i_K(f)}^{(K)}(\X_1)\right|\cdot\left|H_{l_1,\ldots,l_p}(\X_1)\right|\right)}
{l_1!\cdots l_p!}\right|d_N^{-1}\left|\sum_{j=1}^nH_{l_1,\ldots,l_p}(\X_j)\right|\end{multlined}\\
\leq&\begin{multlined}[t][\textwidth-4mm]S_N(n,l_{i_K(f)}^{(K)},u_{i_K(f)}^{(K)})+2nd_N^{-1}E\left(u_{i_K(f)}^{(K)}(\X_1)-l_{i_K(f)}^{(K)}(\X_1)\right)\\
+2\cdot\hspace{-3pt}\sum_{l_1+\ldots+l_p=m}\left|\frac{E\left(\left|u_{i_K(f)}^{(K)}(\X_1)-l_{i_K(f)}^{(K)}(\X_1)\right|\cdot\left|H_{l_1,\ldots,l_p}(\X_1)\right|\right)}
{l_1!\cdots l_p!}\right|\\ \cdot d_N^{-1}\left|\sum_{j=1}^nH_{l_1,\ldots,l_p}(\X_j)\right|\end{multlined}\\
\leq& S_N(n,l_{i_K(f)}^{(K)},u_{i_K(f)}^{(K)})+2nd_N^{-1}2^{-K} +2\cdot2^{-K}d_N^{-1}\cdot\hspace{-3pt}\sum_{l_1+\ldots+l_p=m}\left|\sum_{j=1}^n H_{l_1,\ldots,l_p}(\X_j)\right|,
\end{align*}
by using Cauchy-Schwarz inequality. Since $\e/2+\sum_{k=0}^K\e/(k+3)^2<\e$ for all $K\in\NN$, we have
\begin{align*}
&P^*\left(\sup_{f\in\F}\left|S_N(n,f)\right|>\e\right)\\
\leq &P^*\left(\max_{f\in\F}\left|S_N(n,l_{i_0(f)}^{(0)})\right|>\e/9\right)\\
&+P^*\left(\max_{f\in\F}\left|S_N(n,l_{i_0(f)}^{(0)},l_{i_1(f)}^{(1)})\right|>\e/16\right)+\ldots\\
&+P^*\left(\max_{f\in\F}\left|S_N(n,l_{i_{K-1}(f)}^{(K-1)},l_{i_K(f)}^{(K)})\right|>\e/(K+3)^2\right)\\
&+P^*\left(\max_{f\in\F}\left|S_N(n,l_{i_K(f)}^{(K)},u_{i_K(f)}^{(K)})\right|>\e/(K+4)^2\right)\\
&+P\left(2\cdot2^{-K}d_N^{-1}\cdot\sum_{l_1+\ldots+l_p=m}\left|\sum_{j=1}^n H_{l_1,\ldots,l_p}(\X_j)\right|>\e/2-2nd_N^{-1}2^{-K}\right). \numberthis\label{Diskretisierung}
\end{align*}
Set
\begin{align}
M=&\left|\left\{(l_1,\ldots,l_p)\in\NN^p:l_1+\ldots l_p=m\right\}\right|=\binom{m+p-1}{m}\label{eq multivariate M} \\
K=&\left\lceil\log_2\left(\frac{8Nd_N^{-1}}{\e}\right)\right\rceil\label{K}
\end{align}
and let $N_k$ be the bracketing number with respect to $2^{-k}$. Applying Lemma \ref{lemma 2q-te Momente} yields
\begin{align*}
&P^*\left(\max_{f\in\F}\left|S_N(n,l_{i_{k-1}(f)}^{(k-1)},l_{i_k(f)}^{(k)})\right|>\e/(k+3)^2\right)\\
\leq&\sum_{s=1}^{N_{k-1}}\sum_{u=1}^{N_{k}}P\left(\left|S_N(n,l_s^{(k-1)},l_u^{(k)})\right|>\e/(k+3)^2\right)\\
\leq&\sum_{s=1}^{N_{k-1}}\sum_{u=1}^{N_{k}}(k+3)^{4q}\e^{-2q}E\left|S_N(n,l_s^{(k-1)},l_u^{(k)})\right|^{2q}\\
\leq&C(k+3)^{4q}N_k^2\e^{-2q}(Nd_N^{-1})^{-r}\left(\frac{n}{N}\right)N^{-\gamma}
\end{align*}
for $1\leq k\leq K$. In the same way we get
\begin{align*}
&P^*\left(\max_{f\in\F}\left|S_N(n,l_{i_0(f)}^{(0)})\right|>\e/9\right)\\
\leq&C3^{4q}N_1\e^{-2q}(Nd_N^{-1})^{-r}\left(\frac{n}{N}\right)N^{-\gamma}\\
\intertext{and}
&P^*\left(\max_{f\in\F}\left|S_N(n,l_{i_K(f)}^{(K)},u_{i_K(f)}^{(K)})\right|>\e/(K+4)^2\right)\\
\leq&C(K+4)^{4q}N_{K}\e^{-2q}(Nd_N^{-1})^{-r}\left(\frac{n}{N}\right)N^{-\gamma}.
\end{align*}
Let $M$ be the quantity given by \eqref{eq multivariate M}. Since \eqref{K} implies $(\e/4)^{-2}<2^{2K-2}N^{-2}d_N^{2}$ we have
\begin{align*}
&P\left(2\cdot2^{-K}d_N^{-1}\cdot\sum_{l_1+\ldots+l_p=m}\left|\sum_{j=1}^n H_{l_1,\ldots,l_p}(\X_j)\right|>\e/2-2nd_N^{-1}2^{-K}\right)\\
\leq&P\left(d_N^{-1}\cdot\sum_{l_1+\ldots+l_p=m}\left|\sum_{j=1}^n H_{l_1,\ldots,l_p}(\X_j)\right|>\frac{2^{K-1}\e}{4}\right)\\
\leq&\sum_{l_1+\ldots+l_p=m}P\left(d_N^{-1}\left|\sum_{j=1}^n H_{l_1,\ldots,l_p}(\X_j)\right|>\frac{2^{K-1}\e}{4M}\right)\\
\leq&d_N^{-2}\sum_{l_1+\ldots+l_p=m}E\left|\sum_{j=1}^nH_{l_1,\ldots,l_p}(\X_j)\right|^2\left(\frac{\e}{4}\right)^{-2}2^{-2K+2}M^2\notag\\
\leq&Cd_N^{-2}n^{2-mD}L(bn)^mN^{-2}d_N^2\phantom{\left(\frac{\varepsilon}{4}\right)^{-2}}\notag\\
\leq&C\left(\frac{n}{N}\right)^{2-mD}\left(\frac{L(bn)}{L(N)}\right)^mN^{-mD}L^m(N)\notag\\
\leq&C\left(\frac{n}{N}\right)^{2-mD}N^{-mD+\lambda}
\end{align*}
for any $\lambda>0$. Thus, assumption \eqref{B} and \eqref{K} can be used to bound \eqref{Diskretisierung} finally in the following way
\begin{align*}
&P^*\left(\sup_{f\in\F}\left|S_N(n,f)\right|>\e\right)\\
\leq&C\e^{-2q}(Nd_N^{-1})^{-r}\left(\frac{n}{N}\right)N^{-\gamma}\sum_{k=0}^K(k+3)^{4q}N_k^2
+C\left(\frac{n}{N}\right)^{2-mD}N^{-mD+\lambda}\\
=&C\e^{-2q}(Nd_N^{-1})^{-r}\left(\frac{n}{N}\right)N^{-\gamma}\sum_{k=0}^K2^{-rk}2^{rk}(k+3)^{4q}N_k^2
+C\left(\frac{n}{N}\right)^{2-mD}N^{-mD+\lambda}\\
\leq&C\e^{-2q}(Nd_N^{-1})^{-r}\left(\frac{n}{N}\right)N^{-\gamma}(2^{K})^r(K+3)^{4q}\sum_{k=0}^{\infty}2^{-rk}N_k^2
+C\left(\frac{n}{N}\right)^{2-mD}N^{-mD+\lambda}\\
\leq&C\e^{-2q}(Nd_N^{-1})^{-r}\left(\frac{n}{N}\right)N^{-\gamma}\e^{-r}(Nd_N^{-1})^{r}(K+3)^{4q}
+C\left(\frac{n}{N}\right)^{2-mD}N^{-mD+\lambda}\\
\leq&C\left(\frac{n}{N}\right)N^{-\gamma}(K+3)^{4q}\e^{-(2q+r)}
+C\left(\frac{n}{N}\right)^{2-mD}N^{-mD+\lambda}\\
\leq&CN^{-\rho}\left(\left(\frac{n}{N}\right)\e^{-(2q+r+1)}+\left(\frac{n}{N}\right)^{2-mD}\right)
\end{align*}
for a sufficient small $\rho$. Note,  for the last inequality we use $(K+3)^{4q}\leq C\e^{-1}N^{\delta}$ for any $\delta>0$.
\end{proof}
\begin{proof}[Proof of Theorem \ref{Theorem Reduction Functions}]
If the cross-covariance function satisfies $r^{(i,j)}(k)\leq 1/(4pq-1)$ for all $k\in\NN$ and $1\leq i,j\leq p$ then Lemma \ref{lemma chaining function} holds and Theorem \ref{Theorem Reduction Functions} follows by adapting the proof given by \citet[p. 1781]{Dehling}. The general case can be treated as done by \citet[p. 225]{Taqqu77}. Choose $b\in\NN$ such that $r^{(i,j)}(k)\leq 1/(4pq-1)$ for all $k\geq b$ and set 
\begin{equation*}
\tilde{f}(\X_j)=f(\X_j)-Ef(\X_1)-\sum_{l_1+\ldots+l_p=m}\frac{J_{l_1,\ldots,l_p}(f)}{l_1!\cdots l_p!}H_{l_1,\ldots,l_p}(\X_j).
\end{equation*}
We can decompose $S_N(n,f)$ as follows
\begin{align*}
\left|S_N(n,f)\right|\leq&d_N^{-1}\left(\sum_{j=1}^{j^*}\left|\sum_{k=0}^{\lfloor\frac{n}{b}\rfloor}\tilde{f}(\X_{j+kb})\right|
+\sum_{j=j^*+1}^{b}\left|\sum_{k=0}^{\lfloor\frac{n}{b}\rfloor-1}\tilde{f}(\X_{j+kb})\right|\right)\\
\leq&\sum_{j=1}^b\max_{1\leq l\leq N/b}d_N^{-1}\left|\sum_{k=0}^l\tilde{f}(\X_{j+kb})\right|,
\end{align*}
where $j^*=b$, if $n/b$ is an integer, and $j^*=n-\lfloor n/b\rfloor b$ otherwise. Taking the supremum of both sides yields
\begin{equation*}
\max_{1\leq n\leq N}\sup_{f\in\F}\left|S_N(n,f)\right|\leq\sum_{j=1}^b\max_{1\leq l\leq N}\sup_{f\in\F}d_N^{-1}\left|\sum_{k=1}^l
\tilde{f}(\X_{j+(k-1)b})\right|.
\end{equation*}
For any $j$ the series $(\X_{j+(k-1)b})_{k\geq1}$ satisfies the assumption of Lemma \ref{lemma chaining function}. Therefore,
\begin{align*}
&P^*\left(\max_{1\leq n\leq N}\sup_{f\in\F}\left|S_N(n,f)\right|>\e\right)\\
\leq&\sum_{j=1}^bP^*\left(\max_{1\leq l\leq N}\sup_{f\in\F}d_N^{-1}\left|\sum_{k=1}^l\tilde{f}(\X_{j+(k-1)b})\right|>\e/b\right)\\
\leq&CN^{-\kappa}\left(1+\e^{-(2q+r+1)}\right).
\end{align*}
\end{proof}
\begin{lemma}\label{lemma Runge}
For all $m\in\NN$ and $a_1,\ldots,a_p\in\RR$ with $a_1^2+\ldots+a_p^2=1$ we have
\begin{equation}\label{eq Runge}
H_m\left(\sum_{j=1}^pa_jx_j\right)=\sum_{m_1+\ldots+m_p=m}\frac{m!}{m_1!\cdots m_p!}\prod_{j=1}^pa_j^{m_j}H_{m_j}(x_j).
\end{equation}
\end{lemma}
Since we could not find a proof for this well known result in the literature, see e.g \citet[p. 2255]{Arcones}, \citet[p. 113]{Beran}, we give one here.  
\begin{proof}
We first show that all partial derivatives are equal by using induction. For $m=1$ this is obvious. Since $H_n'(x)=nH_{n-1}(x)$ we have
\begin{align*}
&\frac{\partial}{\partial x_1}
\left(\sum_{m_1+\ldots+m_p=m+1}\frac{(m+1)!}{m_1!\cdots m_p!}\prod_{j=1}^pa_j^{m_j}H_{m_j}(x_j)\right)\\
=&\sum_{m_1+\ldots+m_p=m+1}\frac{(m+1)!}{(m_1-1)!\cdots m_p!}a_1^{m_1}H_{m_1-1}(x_1)\prod_{j=2}^pa_j^{m_j}H_{m_j}(x_j)\\
=&a_1(m+1)\sum_{m_1+\ldots+m_p=m}\frac{(m)!}{m_1!\cdots m_p!}\prod_{j=1}^pa_j^{m_j}H_{m_j}(x_j)\\
=&a_1(m+1)H_m\left(\sum_{j=1}^pa_jx_j\right)\\
=&\frac{\partial}{\partial x_1}H_{m+1}\left(\sum_{j=1}^pa_jx_j\right)
\end{align*}
The other derivatives can be handled similarly. Therefore, \eqref{eq Runge} holds up to a constant. Let $x_1=\ldots=x_p=0$. If $m$ is odd, both sides of \eqref{eq Runge} are equal to zero and thus the constant vanishes. For even $m$ we have $H_m(0)=(-1)^{m/2}(m-1)!!$, where
\begin{equation*}
(m-1)!!:=(m-1)(m-3)\cdots3\cdot1=\frac{m!}{2^{m/2}(m/2)!}.
\end{equation*}
This yields
\begin{align*}
&\sum_{m_1+\ldots+m_p=m}\frac{m!}{m_1!\cdots m_p!}\prod_{j=1}^pa_j^{m_j}H_{m_j}(0)\\
=&\sum_{2m_1+\ldots+2m_p=m}\frac{m!}{(2m_1)!\cdots (2m_p)!}\prod_{j=1}^p(-1)^{m_j}a_j^{2m_j}(2m_j-1)!!\\
=&(-1)^{m/2}\sum_{2m_1+\ldots+2m_p=m}\frac{m!}{(2m_1)!!\cdots (2m_p)!!}\prod_{j=1}^p(a_j^2)^{m_j}\\
=&(-1)^{m/2}\sum_{m_1+\ldots+m_p=m/2}\frac{m!}{2^{m/2}m_1!\cdots m_p!}\prod_{j=1}^p(a_j^2)^{m_j}\\
=&(-1)^{m/2}\sum_{m_1+\ldots+m_p=m/2}\frac{(m-1)!!2^{m/2}(m/2)!}{2^{m/2}m_1!\cdots m_p!}\prod_{j=1}^p(a_j^2)^{m_j}\\
=&(-1)^{m/2}(m-1)!!\sum_{m_1+\ldots+m_p=m/2}\frac{(m/2)!}{m_1!\cdots m_p!}\prod_{j=1}^p(a_j^2)^{m_j}\\
=&H_m(0)\left(\sum_{j=1}^p a_j^2\right)^{m/2}\\
=&H_m(0).
\end{align*}
\end{proof}
\begin{proof}[Proof of Theorem \ref{Theorem Limit Functions}]
By Theorem \ref{Theorem Reduction Functions} it is enough to study the limit of
\begin{equation*}
\left\{d_N^{-1}\sum_{j=1}^{\lfloor Nt\rfloor}\sum_{l_1+\ldots+l_p=m}
\frac{J_{l_1,\ldots,l_p}(f)}{l_1!\cdots l_p!}H_{l_1,\ldots,l_p}(\X_j):(f,t)\in\F\times[0,1]\right\}.
\end{equation*}
We first show that in the current situation Lemma \ref{lemma Runge} can be applied.
By independence of multivariate monomials of degree $m$, we can find for all $k_1,\ldots,k_p$ satisfying $k_1+\ldots+k_p=m$, real numbers $a_{k_1,\ldots,k_p}^{(1)},\ldots a_{k_1,\ldots,k_p}^{(p)}$, such that the $\binom{m+p-1}{m}\times\binom{m+p-1}{m}$ matrix
\begin{equation*}
A=\left(\prod_{i=1}^p(a_{k_1,\ldots,k_p}^{(i)})^{m_i}\right)_{\substack{m_1+\ldots+m_p=m\\k_1+\ldots+k_p=m}}
\end{equation*}
is invertible. More precisely, sorting the tuples $(m_1,\ldots,m_p)$, $m_1+\ldots+m_p=m$ and $(k_1,\ldots,k_p)$, $k_1+\ldots+k_p=m$ lexicographically, the entry $a_{q_1,q_2}$ of $A$ is given by $\prod_{i=1}^p(a_{k_1,\ldots,k_p}^{(i)})^{m_i}$ with respect to this order.
After normalization we have $\sum_{i=1}^p (a_{k_1,\ldots,k_p}^{(i)})^2=1$. For a suitable diagonal matrix $M$ of the same size define $B:=MA^{-1}$ such that
\begin{align*}
&\sum_{k_1+\ldots+k_p=m}b(k_1,\ldots,k_p,l_1,\ldots,l_p)(a_{k_1,\ldots,k_p}^{(1)})^{m_1}\cdots
(a_{k_1,\ldots,k_p}^{(p)})^{m_p}\\
&=\begin{cases}(m!)^{-1}\prod_{i=1}^p l_i!~~~~&\textnormal{ if }(m_1,\ldots,m_p)=(l_1,\ldots,l_p)\\0
&\textnormal{ otherwise,}\end{cases}\numberthis\label{VorbereitungRunge}
\end{align*}
where $b(k_1,\ldots,k_p,l_1,\ldots,l_p)$ denotes that entry of $B$ whose row and column number is given by the lexicographical order of $k_1,\ldots,k_p$ and $l_1,\ldots,l_p$.
Using Lemma \ref{lemma Runge} together with \eqref{VorbereitungRunge} we get
\begin{align*}
&\sum_{\substack{l_1+\ldots+l_p=m\\ k_1+\ldots+k_p=m}}J_{l_1,\ldots,l_p}(f)\left(\prod_{i=1}^p(l_i!)^{-1}\right)
b(k_1,\ldots,k_p,l_1,\ldots,l_p)H_m\left(\sum_{i=1}^p a_{k_1,\ldots,k_p}^{(i)}X_j^{(i)}\right)\\
=&\begin{multlined}[t][\textwidth-20mm]
\sum_{\substack{l_1+\ldots+l_p=m\cr k_1+\ldots+k_p=m}}\sum_{m_1+\ldots+m_p=m}
J_{l_1,\ldots,l_p}(f)\left(\prod_{i=1}^p(l_i!)^{-1}\right)b(k_1,\ldots,k_p,l_1,\ldots,l_p)\\
\times m!\prod_{i=1}^p(m_i!)^{-1}\left(a_{k_1,\ldots,k_p}^{(i)}\right)^{m_i}H_{m_i}\left(X_j^{(i)}\right)
\end{multlined}\\
=&\sum_{l_1+\ldots+l_p=m}J_{l_1,\ldots,l_p}(f)\prod_{i=1}^p(l_i!)^{-1}H_{l_i}\left(X_j^{(i)}\right)
\end{align*}
Define for simplicity
\begin{equation*}
I(f;k_1,\ldots,k_p):=\sum_{l_1+\ldots+l_p=m}J_{l_1,\ldots,l_p}(f)\left(\prod_{i=1}^p(l_i!)^{-1}\right)
b(k_1,\ldots,k_p,l_1,\ldots,l_p),
\end{equation*}
whereby the following identity holds
\begin{align*}
&d_N^{-1}\sum_{j=1}^{\lfloor Nt\rfloor}\sum_{l_1+\ldots+l_p=m}J_{l_1,\ldots,l_p}(f)\prod_{i=1}^p(l_i!)^{-1}H_{l_i}\left(X_j^{(i)}\right)\\
=&d_N^{-1}\sum_{j=1}^{\lfloor Nt\rfloor}\sum_{k_1+\ldots+k_p=m}I(f;k_1,\ldots,k_p)H_m\left(\sum_{i=1}^p 
a_{k_1,\ldots,k_p}^{(i)}X_j^{(i)}\right).
\end{align*}
As stated by \citet[(3.6)]{Arcones}, there exist Hermitian Gaussian random measures $\tilde{B}^{(1)},\ldots, \tilde{B}^{(p)}$, such that
\begin{multline*}
\left\{d_N^{-1}\sum_{j=1}^{\lfloor Nt\rfloor}(H_m(X_j^{(1)}),\ldots,H_m(X_j^{(p)})):0\leq t\leq 1\right\}\\
\xrightarrow{~~~~d~~~~}
\left\{(Z_m^{(1)}(t),\ldots,(Z_m^{(p)}(t)):0\leq t\leq 1\right\},
\end{multline*}
in $(D[0,1])^p$, where the processes $(Z_m^{(j)}(t))$ are up to a constant the Hermite processes defined by \eqref{eq Hermite Prozess Darstellung} with respect to $\tilde{B}^{(j)}$. More precisely, if $\Z=(Z^{(1)},\ldots,Z^{(p)})$ is the vector-valued random spectral measure such that
\begin{equation*}
\X_k=\int_{-\pi}^{\pi}e^{ikx}\Z(dx),
\end{equation*}
see e.g. \citet[Theorem 11.8.2]{Brockwell}, then for any bounded symmetric intervals $A_1,\ldots,A_p\subset\RR$, the random vector $(\tilde{B}^{(1)}(A_1),\ldots, \tilde{B}^{(p)}(A_p))$ is the weak limit of \linebreak$L^{-1/2}(N)N^{D/2}(Z^{(1)}(N^{-1}A_1),\ldots,Z^{(p)}(N^{-1}A_p))$, see \citet[p. 1172]{Ho}.\\
Note that for any integrable function $h$
\begin{align*}
&\int_{\RR^m}^{''}h(x_1,\ldots,x_m)
\left(\sum_{i=1}^p a_{k_1,\ldots,k_p}^{(i)}\tilde{B}^{(i)}\right)(dx_1)\cdots
\cdots \left(\sum_{i=1}^p a_{k_1,\ldots,k_p}^{(i)}\tilde{B}^{(i)}\right)(dx_m)
\\
=&\sum_{j_1,\ldots,j_m=1}^p a_{k_1,\ldots,k_p}^{(j_1)}\cdots a_{k_1,\ldots,k_p}^{(j_m)}
\int_{\RR^m}^{''}h(x_1,\ldots,x_m)
\tilde{B}^{(j_1)}(dx_1)\cdots \tilde{B}^{(j_m)}(dx_m).
\end{align*}
and therefore
\begin{multline}\label{eq multivarite beweis limit 2 almost sure vorbereitung}
\left\{d_N^{-1}\sum_{j=1}^{\lfloor Nt\rfloor}H_m(\sum_{i=1}^p a_{k_1,\ldots,k_p}^{(i)}X_j^{(i)}):k_1,\ldots,k_p=m, t\in[0,1]\right\}\xrightarrow{~~~~d~~~~}\\
\left\{\sum_{j_1,\ldots,j_m=1}^p a_{k_1,\ldots,k_p}^{(j_1)}\cdots a_{k_1,\ldots,k_p}^{(j_m)}
Z_{j_1,\ldots,j_m}(t):k_1,\ldots,k_p=m, t\in[0,1]\right\},
\end{multline}
where
\begin{align*}
&Z_{j_1,\ldots,j_m}(t)\\
=&\tilde{K}_{j_1,\ldots,j_m}(m,D)
 \int_{\RR^m}^{''}\frac{e^{it(x_1+\ldots+x_m)}-1}{i(x_1+\ldots+x_m)}\prod_{j=1}^m|x_j|^{-(1-D)/2}\tilde{B}^{(j_1)}(dx_1)\cdots \tilde{B}^{(j_m)}(dx_m)
\end{align*}
and $\tilde{K}_{j_1,\ldots,j_m}(m,D)$ is a suitable constant. Using the almost sure representation theorem, see \citet[p. 71]{Pollard}, we can find suitable versions $(\tilde{S}_N(t))_{N\geq1}$ and $(\tilde{Z}(t))$ such that this convergence holds almost surely. Since the functions $I(f;k_1,\ldots,k_p)$, $k_1+\ldots+k_p=m$, are bounded we have
\begin{multline*}
\left\{d_N^{-1}I(f;k_1,\ldots,k_p)\tilde{S}_N^{(k_1,\ldots,k_p)}(t):k_1+\ldots+k_p=m, f\in\F, t\in[0,1]\right\}\xrightarrow{~~~~a.s.~~~~}\\
\left\{I(f;k_1,\ldots,k_p)\tilde{Z}^{(k_1,\ldots,k_p)}(t):k_1+\ldots+k_p=m, f\in\F, t\in[0,1]\right\}
\end{multline*}
and consequently
\begin{align*}
\begin{multlined}[t][\textwidth-4mm]
\Biggl\{d_N^{-1}\sum_{j=1}^{\lfloor Nt\rfloor}\sum_{k_1+\ldots+k_p=m}I(f;k_1,\ldots,k_p)H_m\left(\sum_{i=1}^p 
a_{k_1,\ldots,k_p}^{(i)}X_j^{(i)}\right): f\in\F, t\in[0,1]\Biggr\}\\
\xrightarrow{~~~~d~~~~}
\Biggl\{\sum_{j_1,\ldots,j_m=1}^p \sum_{k_1+\ldots+k_p=m}a_{k_1,\ldots,k_p}^{(j_1)}\cdots a_{k_1,\ldots,k_p}^{(j_m)}
I(f;k_1,\ldots,k_p)Z_{j_1,\ldots,j_m}(t):\\ f\in\F, t\in[0,1]\Biggr\}
\end{multlined}
\end{align*}
in $\l^{\infty}(\F\times[0,1])$. To conclude we have to verify that the limiting process is the one stated in Theorem \ref{Theorem Limit Functions}.
\begin{align*}
&\sum_{j_1,\ldots,j_m=1}^p\sum_{k_1+\ldots+k_p=m}
I(f;k_1,\ldots,k_p) a_{k_1,\ldots,k_p}^{(j_1)}\cdots a_{k_1,\ldots,k_p}^{(j_m)}\\
=&\sum_{j_1,\ldots,j_m=1}^p\sum_{\substack{k_1+\ldots+k_p=m\\l_1+\ldots+l_p=m}}
J_{l_1,\ldots,l_p}(f)\left(\prod_{i=1}^p(l_i!)^{-1}\right)
b(k_1,\ldots,k_p,l_1,\ldots,l_p)a_{k_1,\ldots,k_p}^{(j_1)}\cdots a_{k_1,\ldots,k_p}^{(j_m)}\\
=&\begin{cases}(m!)^{-1}J_{l_1,\ldots,l_p}(f)~~~~&\textnormal{ if }l_i=i(j_1,\ldots,j_m):=\left|\{j_u=i:u=1,\ldots,m\}\right|\\0
&\textnormal{ otherwise.}\end{cases}
\end{align*}
At this point we can close the proof since for $l_i=i(j_1,\ldots,j_m)$, $1\leq i\leq p$, we have
\begin{equation*}
\tilde{J}_{j_1,\ldots,j_m}(f)=(m!)^{-1}J_{l_1,\ldots,l_p}(f).
\end{equation*}
\end{proof}
\section*{Acknowledgements}
The author's research was supported by Collaborative Research Center SFB 823 \textit{Statistical modeling of nonlinear dynamic processes}.
\newpage
\bibliography{literatur}
\end{document}